\documentclass[12pt]{amsart}
\usepackage{mathrsfs}
\usepackage{verbatim}
\usepackage{hyperref}
\usepackage{color}
\usepackage{amsfonts}
\usepackage{amsmath}
\usepackage{amssymb}
\usepackage{graphicx}
\usepackage{subfigure}

\allowdisplaybreaks

\begin{document}
\newtheorem{theorem}{Theorem}
\newtheorem{proposition}[theorem]{Proposition}
\newtheorem{conjecture}[theorem]{Conjecture}
\def\theconjecture{\unskip}
\newtheorem{corollary}[theorem]{Corollary}
\newtheorem{lemma}[theorem]{Lemma}
\newtheorem{sublemma}[theorem]{Sublemma}
\newtheorem{observation}[theorem]{Observation}
\theoremstyle{definition}
\newtheorem{definition}{Definition}
\newtheorem{notation}[definition]{Notation}
\newtheorem{remark}[definition]{Remark}
\newtheorem{question}[definition]{Question}
\newtheorem{questions}[definition]{Questions}
\newtheorem{example}[definition]{Example}
\newtheorem{problem}[definition]{Problem}
\newtheorem{exercise}[definition]{Exercise}

\numberwithin{theorem}{section} \numberwithin{definition}{section}
\numberwithin{equation}{section}

\def\earrow{{\mathbf e}}
\def\rarrow{{\mathbf r}}
\def\uarrow{{\mathbf u}}
\def\varrow{{\mathbf V}}
\def\tpar{T_{\rm par}}
\def\apar{A_{\rm par}}

\def\reals{{\mathbb R}}
\def\torus{{\mathbb T}}
\def\heis{{\mathbb H}}
\def\integers{{\mathbb Z}}
\def\naturals{{\mathbb N}}
\def\complex{{\mathbb C}\/}
\def\distance{\operatorname{distance}\,}
\def\support{\operatorname{support}\,}
\def\dist{\operatorname{dist}\,}
\def\Span{\operatorname{span}\,}
\def\degree{\operatorname{degree}\,}
\def\kernel{\operatorname{kernel}\,}
\def\dim{\operatorname{dim}\,}
\def\codim{\operatorname{codim}}
\def\trace{\operatorname{trace\,}}
\def\Span{\operatorname{span}\,}
\def\dimension{\operatorname{dimension}\,}
\def\codimension{\operatorname{codimension}\,}
\def\nullspace{\scriptk}
\def\kernel{\operatorname{Ker}}
\def\ZZ{ {\mathbb Z} }
\def\p{\partial}
\def\rp{{ ^{-1} }}
\def\Re{\operatorname{Re\,} }
\def\Im{\operatorname{Im\,} }
\def\ov{\overline}
\def\eps{\varepsilon}
\def\lt{L^2}
\def\diver{\operatorname{div}}
\def\curl{\operatorname{curl}}
\def\etta{\eta}
\newcommand{\norm}[1]{ \|  #1 \|}
\def\expect{\mathbb E}
\def\bull{$\bullet$\ }

\def\xone{x_1}
\def\xtwo{x_2}
\def\xq{x_2+x_1^2}
\newcommand{\abr}[1]{ \langle  #1 \rangle}

\newcommand{\Norm}[1]{ \left\|  #1 \right\| }
\newcommand{\set}[1]{ \left\{ #1 \right\} }
\def\one{\mathbf 1}
\def\whole{\mathbf V}
\newcommand{\modulo}[2]{[#1]_{#2}}
\def \essinf{\mathop{\rm essinf}}
\def\scriptf{{\mathcal F}}
\def\scriptg{{\mathcal G}}
\def\scriptm{{\mathcal M}}
\def\scriptb{{\mathcal B}}
\def\scriptc{{\mathcal C}}
\def\scriptt{{\mathcal T}}
\def\scripti{{\mathcal I}}
\def\scripte{{\mathcal E}}
\def\scriptv{{\mathcal V}}
\def\scriptw{{\mathcal W}}
\def\scriptu{{\mathcal U}}
\def\scriptS{{\mathcal S}}
\def\scripta{{\mathcal A}}
\def\scriptr{{\mathcal R}}
\def\scripto{{\mathcal O}}
\def\scripth{{\mathcal H}}
\def\scriptd{{\mathcal D}}
\def\scriptl{{\mathcal L}}
\def\scriptn{{\mathcal N}}
\def\scriptp{{\mathcal P}}
\def\scriptk{{\mathcal K}}
\def\frakv{{\mathfrak V}}
\def\C{\mathbb{C}}
\def\R{\mathbb{R}}
\def\Rn{{\mathbb{R}^n}}
\def\Sn{{{S}^{n-1}}}
\def\M{\mathbb{M}}
\def\N{\mathbb{N}}
\def\Q{{\mathbb{Q}}}
\def\Z{\mathbb{Z}}
\def\F{\mathcal{F}}
\def\L{\mathcal{L}}
\def\S{\mathcal{S}}
\def\supp{\operatorname{supp}}
\def\dist{\operatorname{dist}}
\def\essi{\operatornamewithlimits{ess\,inf}}
\def\esss{\operatornamewithlimits{ess\,sup}}
\author{Heng Gu}
\address{Heng Gu
\\
School of Mathematical Sciences
\\
Beijing Normal University
\\
Laboratory of Mathematics and Complex Systems
\\
Ministry of Education
\\
Beijing 100875
\\
People's Republic of China
}
\email{henggu@mail.bnu.edu.cn}

\author{Qingying Xue}
\address{
        Qingying Xue\\
        School of Mathematical Sciences\\
        Beijing Normal University \\
        Laboratory of Mathematics and Complex Systems\\
        Ministry of Education\\
        Beijing 100875\\
        People's Republic of China}
\email{qyxue@bnu.edu.cn}
\author{K\^{o}z\^{o} Yabuta}
\address{
        K\^oz\^o Yabuta\\
        Research Center for Mathematical Sciences\\
         Kwansei Gakuin University\\
         Gakuen 2-1\\
         Sanda 669-1337\\
        Japan}
\email{kyabuta3@kwansei.ac.jp}
\thanks{The authors were supported partly by NSFC
(No. 11471041 and 11671039), the Fundamental Research Funds for the Central Universities (NO. 2014KJJCA10) and NCET-13-0065. The third author
was supported partly by Grant-in-Aid for Scientific Research (C)
Nr. 15K04942, Japan Society for the Promotion of Science.\\ \indent Corresponding
author: Qingying Xue\indent Email: qyxue@bnu.edu.cn}
\title[On Some properties of dyadic operators ]{\textbf{ On Some properties of dyadic operators}}
\date{\today}
\maketitle

\begin{abstract}
In this paper, the objects of our investigation are some dyadic operators, 
including dyadic shifts, multilinear paraproducts and multilinear Haar
 multipliers. We mainly focus on the continuity and compactness of these 
 operators. First, we consider the continuity properties of these operators. 
 Then, by the Fr\'{e}chet-Kolmogorov-Riesz-Tsuji theorem, the 
 non-compactness properties of these dyadic operators will be studied. 
 Moreover, we show that their commutators are compact with \textit{CMO} 
 functions, which is quite different from the non-compaceness properties of 
 these dyadic operators. 
 These results are similar to those for Calder\'on-Zygmund  singular integral 
 operators. 
\end{abstract}


\section{Introduction}
It is well known that the dyadic operators, such as paraproducts, Haar 
multipliers and dyadic shifts, play very important roles in Harmonic Analysis. 
The study of paraproducts may be traced back to the famous work of Bony in 
\cite{MB}. Since then, many works had been done in this field. Among those 
achievements is the celebrated work of David and Journ\'{e} \cite{DJ}. 
Using the techniques of paraproducts, David and Journ\'{e} established the 
$T(1)$ theorem and thus gave a boundedness criterion for generalized 
Calder\'{o}n-Zygmund operators. 
The investigation of Haar multipliers may be dated back to the 
$A_2$ conjecture for Haar multipliers consider by Wittwer in \cite{W}. 
Subsequently, using the combination of Bellman function technique and heat 
extension, Petermichl and Volberg extended the same result to Beurling-Ahlfors 
transforms in \cite{PV}. As for the dyadic shifts, it is known that an 
elementary dyadic shift with parameter $(m, n)$ ($m,n\in\mathbb{N}$) is 
an operator given by
\begin{equation}\aligned
\mathbb Sf(x)&=\sum_{I\in \mathcal{D}}\frac{1}{|I|}\int_{I}a_I(x,y)f(y)dy
=\sum_{I\in \mathcal{D}}
\sum_{\substack{I^{'},I^{''}\in\mathcal{D},I^{'},I^{''}\subset I
\\l(I^{'})=2^{-m}l(I)\\ l(I^{''})=2^{-n}l(I)}}
\frac{1}{|I|}\langle f,h_{I^{'}}\rangle h_{I^{''}}
\endaligned
\end{equation}
where $h_{I^{'}}$ and $h_{I^{''}}$ are Haar functions for the cubes $I^{'}$ 
and $I^{''}$ respectively in $\mathbb{R}^{d}$, subject to normalization
$
\|h_{I^{'}}\|_{\infty}\cdot\|h_{I^{''}}\|_{\infty}\leq 1
$
and
\begin{align*}
a_{I}=\sum_{\substack{I^{'},I^{''}\in\mathcal{D},I^{'},I^{''}\subset I
\\l(I^{'})=2^{-m}l(I), l(I^{''})=2^{-n}l(I)}}h_{I^{'}}(y)h_{I^{''}}(x)
\end{align*}
The number $r=\max(m, n)$ is called the complexity of the dyadic shift. 
There are two important works in the earlier stage of investigation. The first 
one is given in \cite{NTV} which concerned with the boundedness of dyadic 
shifts.
The second one is given by Lacey, Petermichl and Reguera \cite{LPR} which 
demonstrates the $A_2$ conjecture for general dyadic shifts. A recent nice 
work  \cite{HPTV} states that an arbitrary Calder\'{o}n-Zygmund operator can 
be presented as an average of random dyadic shifts and random dyadic 
paraproducts. This demonstrates the importance of the dyadic shifts and 
people are beginning to pay more attention to these operators. 

Still more recently, the following multilinear dyadic paraproducts 
$\pi_{b}^{\vec{\alpha}}$, Haar multipliers $ P^{\vec{\alpha}}$ and 
$T_{\epsilon}^{\vec{\alpha}}$ have been introduced and studied by 
Kunwar \cite{IK1}.
\begin{align}
\pi_{b}^{\vec{\alpha}}(\vec{f})(x
)=\sum_{I\in D}\frac{\langle b,h_I\rangle}{|I|}\biggl(
\prod_{j=1}^{m}\frac{\langle f_j, h_{I}^{1+\alpha_{j}}\rangle}{|I|}\biggr) 
h_{I}^{1+\sigma(\vec{\alpha})}, \quad 
\vec{\alpha}=(\alpha_1,\cdots,\alpha_m)\in\{0,1\}^m,
\end{align}
\begin{align}
 P^{\vec{\alpha}}(\vec{f})(x)=\sum_{I\in D}\biggl(
 \prod_{j=1}^{m}\frac{\langle f_j, h_{I}^{1+\alpha_{j}}\rangle}{|I|}\biggr) 
 h_{I}^{\sigma(\vec{\alpha})}, \quad 
 \vec{\alpha}\in\{0,1\}^m\setminus\{1,\cdots,1\},\end{align}
\begin{align}
T_{\epsilon}^{\vec{\alpha}}(\vec{f})(x)
=\sum_{I\in \mathcal{D}}\epsilon_{I}\biggl(
\prod_{j=1}^{m}\frac{\langle f_j, h_{I}^{1+\alpha_{j}}\rangle}{|I|}\biggr) 
h_{I}^{\sigma(\vec{\alpha})}, \quad 
\vec{\alpha}\in\{0,1\}^m\setminus\{1,\cdots,1\},
\end{align}
where
$b\in \textit{BMO}^d$, and 
$\epsilon=\{\epsilon_I\}_{I\in\mathcal{D}} $ is bounded 
and $\sigma(\vec{\alpha})$ is denoted to be the number of 0 components in 
$\vec{\alpha}$. 

In \cite{IK1}, Kunwar investigated the strong and weak type boundedness 
properties of $\pi_{b}^{\vec{\alpha}}$ and its commutators.
Moreover, Kunwar  \cite{IK1}  demonstrated that
\begin{align*}
f_1\cdots f_m=\sum_{\vec{\alpha}\in\{0,1\}^m\backslash\{(1,\cdots,1)\}}
P^{\vec{\alpha}}(\vec{f}),\quad \hbox{for}\ f_j\in L^{p_j}(\mathbb{R}).
\end{align*}
If $1<p_1,\cdots,p_m<\infty$ with 
$\frac{1}{p}=\sum_{i=1}^m\frac{1}{p_i}$ and $b\in \textit{BMO}^d$,  
Kunwar \cite{IK2} showed that the Haar 
multipliers and their commutators enjoy the properties that
\begin{align*}
T_{\epsilon}^{\vec{\alpha}}:
 L^{p_1}\times L^{p_2}\times\cdots\times L^{p_m}\rightarrow L^{p}
\end{align*}
and
\begin{align*}
[b, T_{\epsilon}^{\vec{\alpha}}]_j: 
L^{p_1}\times L^{p_2}\times\cdots\times L^{p_m}\rightarrow L^{p}, 
\quad \hbox{for}\ j=1,\cdots, m,
\end{align*}
where $[b, T_{\epsilon}^{\vec{\alpha}}]_j$ is denoted to be the commutator of 
$T_{\epsilon}^{\vec{\alpha}}$ in the $j$-th entry. 

This paper will be devoted to investigated the continuity and compactness of the above dyadic type operators, including their commutators. First, we consider the continuity properties of them and get the following result.\par
\begin{theorem} [Continuity of dyadic operators]\label{Continuity} 
The following statements hold:\par
\noindent$\mathrm{(i)}$ Let $|\nabla f|\in L^{\infty}(\mathbb{R}^d)$. 
Then $\mathbb Sf(x)$ is almost everywhere continuous.\par
\noindent$\mathrm{(ii)}$ 
Let $\vec{\alpha}\in \{0,1\}^m\setminus\{(1,\cdots,1)\}$ and 
$\epsilon=\{\epsilon_{I}\} $ be bounded sequence.  Suppose that $f_j^{'}$ is 
bounded when $\alpha_j=0$ and $f_j$ is bounded when $\alpha_j=1$ in 
$\mathbb{R}$. Then $\pi_{b}^{\vec{\alpha}}(\vec{f})(x)$ and 
$T_{\epsilon}^{\vec{\alpha}}(\vec{f})(x)$ are almost everywhere continuous.
\end{theorem}

\noindent{\bf {Remark 1.1.}}
For dyadic paraproducts $\pi_{b}^{\vec{\alpha}}(\vec{f})(x)$, when 
$\vec{\alpha}=\{(1,\cdots,1)\}$, then $\pi_{b}^{\vec{\alpha}}(\vec{f})(x)$ is 
also almost everywhere continuous if $b^{'}(x)$ is bounded and for all $f_j$ 
is bounded in $\mathbb{R}$. The square of the Littlewood-Paley square function 
$Sf(x)=\Big(\sum_{I\in\mathcal{D}}\big(\frac{\langle f,h_I\rangle}{|I|}\big)^2
\chi_{I}\Big)^{1/2}$ and Haar multipliers $P^{\vec{\alpha}}(\vec{f})(x)$ are 
special cases of $T_{\epsilon}^{\vec{\alpha}}(\vec{f})(x)$. Therefore, 
they are also almost everywhere continuous.\par
There are many results about the compactness of the non-dyadic operators. 
For example, \cite{AU} and \cite{WS} are some nice works in the earlier stage. Recently, the authors in \cite{BT}, \cite{DMX} studied the compactness of 
bilinear operators and their commutators. But there is no compactness or non-compactness results for dyadic operators. Thus, it is quite natural to ask whether these dyadic operators are compact or not. Below, we will give a negative answer to this question. 
\begin{theorem}[Noncompactness of dyadic operators]\label{noncompaceness}
\noindent$\mathrm{(i)}$ 
Let $\epsilon=\{\epsilon_{I}\} $ be a bounded sequence and suppose that
there exists a constant $A>0$ such that 
$\#\{I\in\mathcal D: |\epsilon_I|\ge A\}=\infty$. Let
$\frac1{p}=\frac1{p_1}+\cdots+\frac1{p_m}$ with 
$1<{p_1},\cdots,{p_m}<\infty$. Then $T_{\epsilon}^{\vec{\alpha}}$ 
is not a compact operator from $L^{p_1}(\mathbb{R})\times\cdots\times 
L^{p_m}(\mathbb{R})$ to $L^p(\mathbb{R})$ for 
$\vec{\alpha}\in \{0,1\}^m\setminus\{(1,\cdots,1)\}$.\par
\noindent$\mathrm{(ii)}$ 
Let $m,n\in\mathbb{N}$ and 
suppose that there exists a constant $A>0$ such that 
\begin{align*}
&\#\{I\in\mathcal D: A\le\|h_{I^{'}}\|_{\infty}\cdot\|h_{I^{''}}\|_{\infty}
\leq 1\text{ for some }I^{'},I^{''}\in\mathcal{D},I^{'},I^{''}\subset I,
\\
&\hspace{7cm}l(I^{'})=2^{-m}l(I),l(I^{''})=2^{-n}l(I)\}=\infty.
\end{align*}
Then, 
dyadic shift with parameters $(m, n)$ is not a compact operator. 
\end{theorem}
There also exists $b\in L^\infty\subset \textit{BMO}$ such that 
$\pi_{b}^{\vec \alpha}$ is not a compact operator. However, 
for $b\in \textit{CMO}$, it can be shown that 
$\pi_{b}^{\vec{\alpha}}$ is a compact operator. Consequently, we get
\begin{theorem}[Compactness of $\pi_{b}^{\vec{\alpha}}$]\label{compactness-pi}
Let $b\in \textit{CMO}$ and $\frac1{p}=\frac1{p_1}+\cdots+\frac1{p_m}$ with 
$1<{p_1},\cdots,{p_m}<\infty$. Then $\pi_{b}^{\vec{\alpha}}$ is a compact 
operator from $L^{p_1}(\mathbb{R})\times\cdots\times L^{p_m}(\mathbb{R})$ to 
$L^p(\mathbb{R})$ for $\vec{\alpha}\in \{0,1\}^m$.
\end{theorem}
Nevertheless, like in \cite{BT} and \cite{DMX} for many non-dyadic operators, 
they may be not compact but their commutators and iterated commutators can be 
compact. Therefore, we try to figure out whether the commutators and the iterated 
commutators of these dyadic operators are compact or not. First, following the usual definition of commutators $[b,T](f)=bT(f)-T(bf)$, we define the iterated commutators
of Haar multipliers $T_{\epsilon,\Pi\textbf{b}}^{\vec{\alpha}}$ by
\begin{align}
T_{\epsilon,\Pi\textbf{b}}^{\vec{\alpha}}(\vec{f})
&=[b_1,[b_2,\cdots[b_{m-1}[b_m, 
T_{\epsilon}^{\vec{\alpha}}]_m]_{m-1}\cdots]_2]_1(\vec{f}).
\end{align}
 We formulate the results for the compactness of the commutators as follows:
\begin{theorem}[Compactness of commutators]\label{compaceness-commutators}
Let $\epsilon=\{\epsilon_{I}\} $ be a bounded sequence and 
$\frac1{p}=\frac1{p_1}+\cdots+\frac1{p_m}$ with 
$1<{p_1},\cdots,{p_m}<\infty$. The following statements hold:\par
\noindent$\mathrm{(i)}$ 
Let $b\in \textit{CMO}$. Then  $[b,T_{\epsilon}^{\vec{\alpha}}]_{i}$ is a compact 
operator from $L^{p_1}(\mathbb{R})\times\cdots\times L^{p_m}(\mathbb{R})$ to 
$L^p(\mathbb{R})$ for all $\vec{\alpha}\in \{0,1\}^m\setminus\{(1,\cdots,1)\}$ 
and $1\leq i\leq m$.\par
\noindent$\mathrm{(ii)}$ 
Let $\vec{b}=(b_1, \cdots, b_m)\in \textit{CMO}^m$. Then 
$T_{\epsilon,\Pi\textbf{b}}^{\vec{\alpha}}$ is a compact operator from 
$L^{p_1}(\mathbb{R})\times\cdots\times L^{p_m}(\mathbb{R})$ to 
$L^p(\mathbb{R})$ for $\vec{\alpha}\in \{0,1\}^m\setminus\{1, \cdots, 1\}$.\par
\noindent$\mathrm{(iii)}$ Let $b\in CMO$. Then $[b,\,\mathbb S]$ is a compact operator from $L^{p}(\mathbb{R}^d)$ to $L^p(\mathbb{R}^d)$. \par
\end{theorem}
The rest of this article is organized as follows. Some preliminaries which will be used later are given in Section \ref{pre}. The proof of Theorem \ref{Continuity} will be given in Section $\ref{TH1}$. Section $\ref{non-c}$ will be devoted to demonstrate Theorem \ref{noncompaceness} and Theorem \ref{compactness-pi}. The proof of Theorem $\ref{compaceness-commutators}$ will be presented in Section \ref{TH5}. 

\vspace{0.5cm}
\section{Preliminaries}\label{pre}
\noindent
2.1\quad \textbf{Standard dyadic lattices and Haar system.}
The standard dyadic system in $\mathbb{R}^{d}$ is
\begin{align*}
\mathcal{D}:=\bigcup_{k\in\mathbb{Z}}\mathcal{D}_{k}, \quad \mathcal{D}_{k}
:=\{2^k([0,1)^{d}+m):m \in\mathbb{Z}^{d}\}.
\end{align*}
For $I\in \mathcal{D}$, $I^{(j)}$ is denoted to be the $j$-th dyadic ancestor of $I$ 
($2^{j}l(I)=l(I^{(j)})$ and $I\subset I^{(j)}$). Given a cube $I=x+[0,1)^{d}$, let $ch(I):=\{x+\eta l/2+[0,1/2)^{d}: \eta\in\{0,1\}^d\} $be the collection of 
dyadic children of $I$. 
Thus $\mathcal{D}_{k-1}=\bigcup\{ch(I): I\in\mathcal{D}_{k}\}$. Associated to 
the dyadic cube $I$ there is a Haar function $h_I$ which is defined by
\begin{align*}
h_I=\sum_{J\in \{ch(I)\}}\alpha_{J}1_{J},\quad  \sum_{J\in \{ch(I)\}}\alpha_{J}|J|=0.
\end{align*}
When $I$ is a dyadic interval and let $I_{+}$ and $I_{-}$ be the right and left halves of $I$, then , the Haar function $h_I$ is defined by
$
h_I=1_{I_{+}}-1_{I_{-}}.
$ It is well 
known that the collection of all Haar functions 
$\{\frac{h_I}{\sqrt{|I|}}: I\in\mathcal{D}\}$ is an orthonormal basis of 
$L^2({\mathbb{R}})$ and an unconditional basis of $L^p({\mathbb{R}})$ for 
$1<p<\infty$.\par

\vspace{0.4cm}
\noindent
2.2\quad \textbf{Multilinear weights.}\par
Following the notation in \cite{LOPTT}, for $m$ exponents $p_1, \cdots, p_m$, we write $p$ for the number given by $1/p=1/p_1+\cdots+1/p_m$ and $\vec{P}$ for the vector $\vec{P}=(p_1, \cdots, p_m)$. \begin{definition}[Multiple weights, \cite{LOPTT}] For $1\leq p_1, \cdots, p_m<\infty$ and a multiple weight $\vec{\omega}=(\omega_1, \cdots, \omega_m)$, we say that $\vec{\omega}$ satisfies the multilinear $A_{\vec{P}}$ condition if
\begin{align*}
\sup_I\left(\frac{1}{|I|}\int_I\nu_{\vec{\omega}}\right)^{1/p}\prod_{j=1}^m\left(\frac{1}{|I|}\int_I\omega_j^{1-p_j^{'}}\right)^{1/p_j}<\infty,
\end{align*}
where $\nu_{\vec{\omega}}=\prod_{j=1}^m\omega_j^{p/p_j}$. When $p_j=1$, $\left(\frac{1}{|I|}\int_I\omega_j^{1-p_j^{'}}\right)^{1/p_j}$ is understood as $\|\omega_j^{-1}\|_{L^{\infty}(I)}$.
\end{definition}
By H\"{o}lder's inequality, it is easy to see that
\begin{align*}
\prod_{j=1}^mA_{\vec{P_j}}\subset A_{\vec{P}}.
\end{align*}
Moreover, if $\omega\in A_{\vec{P}}$, then we have $\nu_{\vec{\omega}}\in A_{mp}$. We will similarly denote the dyadic multilinear $A_{\vec{P}}$ class by $A_{\vec{P}}^d$.


\vspace{0.4cm}
\noindent
2.3\quad \textbf{BMO space.}
For a locally integrable function $b$ on $\mathbb{R}$, set
\begin{align*}
\big\|b\big\|_{\textit{BMO}}
=\sup_{I}\frac{1}{|I|}\int_{I}|b(x)-\langle b\rangle_I|dx,
\end{align*}
where the supremum is taken over all intervals $I$ in $\mathbb{R}$. 
The function $b$ is called of bounded mean oscillation if 
$\|b\|_{\textit{BMO}}<\infty$ and $\textit{BMO}(\mathbb{R})$ is the set of all 
locally integrable functions $b$ on $\mathbb{R}$ with 
$\|b\|_{\textit{BMO}}<\infty$. We define $\textit{CMO}$ to be the closure of 
${C}_{c}^{\infty}$ in the $\textit{BMO}$ norm.\par
\vspace{0.3cm}
\noindent If we take the supremum over all dyadic intervals in $\mathbb{R}$, 
we get a larger space of dyadic $\textit{BMO}$ functions which is denoted by 
$\textit{BMO}^d$.
\noindent For $1<r<\infty$, define
\begin{align*}
\textit{BMO}_r=\{b\in L_{loc}^{p}: \big\|b\big\|_{\textit{BMO}_r}<\infty\},
\end{align*}
where $\|b\|_{\textit{BMO}_r}:=\big(
\sup_{I}\frac{1}{|I|}\int_{I}|b(x)-\langle b\rangle_I|^rdx\big)^{\frac{1}{r}}$.
\noindent For any $1<r<\infty$, the norms $\|b\|_{\textit{BMO}_r}$ and 
$\|b\|_{\textit{BMO}}$ are equivalent (see \cite{H}, \cite{JN}). For $r=2$, it
 follows frow the orthogonality of Haar system that 
\begin{align*}
\big\|b\big\|_{\textit{BMO}_2^d}
=\bigg(\sup_{I}\frac{1}{|I|}\sum_{J\subseteq I}
\frac{\langle b,h_J\rangle^2}{|J|^2}\bigg)^{1/2}.
\end{align*}
On $\mathbb R^d$, we may define $\textit{BMO}(\mathbb R^d)$ and its dyadic version 
in a similar way. 

\vspace{0.4cm}
\noindent
2.4\quad \textbf{A key lemma.}
The following lemma is quite useful and it provides a foundation for our analysis in the proof. \par
\begin{lemma}[Fr\'{e}chet-Kolmogorov-Riesz-Tsuji theorem,
 \cite{Tsuji, Y})].
Let $0< r<\infty$. A closed subset $\mathcal{K}\subseteq L^{r}$ is compact
if and only if the following three conditions are satisfied:\par
(a) $\mathcal{K}$ is boundedness in $L^{r}$;\par
(b) $\lim_{A\to\infty}\int_{|x|>A}|f(x)|^rdx=0$ uniformly for
$f\in\mathcal{K}$;\par
(c) $\lim_{t\to 0}\big\|f(x+t)-f(x)\big\|_{L^r}=0$ uniformly for
$f\in\mathcal{K}$.\end{lemma}

\vspace{0.5cm}

\section{Proof of Theorem \ref{Continuity}}\label{TH1}
Now, we begin to prove Theorem \ref{Continuity}.
\proof
\textbf{(i)}  Our first aim is to demonstrates the continuity of $\mathbb S(f)$.
Let $|\nabla f|$ be bounded in $\mathbb{R}^{d}$. 
For $\varepsilon>0$, there exists $k_0>0$ such that 
$\sum_{k=k_0}^{\infty}\frac{1}{2^k}<\varepsilon$. Then , it holds that\begin{align*}
\lim_{t\rightarrow0}|\mathbb Sf(x+t)-\mathbb Sf(x)|
\leq I_1+{I_2}
\end{align*}
where
\begin{align*}
{I_1}
=\lim_{t\rightarrow0}\Big|\sum_{l(I)\leq 2^{-k_0}}\frac{1}{|I|}
\int_{I}a_I(x+t,y)f(y)dy-\sum_{l(I)
\leq 2^{-k_0}}\frac{1}{|I|}\int_{I}a_I(x,y)f(y)dy\Big|.
\end{align*}

\begin{align*}
{I_2}=\lim_{t\rightarrow0}\Big|\sum_{l(I)> 2^{-k_0}}\frac{1}{|I|}
\int_{I}a_I(x+t,y)f(y)dy-\sum_{l(I)> 2^{-k_0}}\frac{1}{|I|}
\int_{I}a_I(x,y)f(y)dy\Big|.
\end{align*}
Therefore, we need to consider the contributions of  ${I_1}$ and ${I_2}$, respectively.\par
\vspace{0.4cm}
\noindent (1) Estimates for ${I_1}$.
For any $x\in I$, there is only one cube $I^{''}$ such that $x\in I^{''}$. 
Hence, noting that $\|h_{I^{'}}\|_{\infty}\|h_{I^{''}}\|_{\infty}\leq1$, 
it yields that
\begin{align*}
\bigg|\frac{1}{|I|}\int_{I}a_{I}(x,y)f(y)dy\bigg|
&=\bigg|\frac{1}{|I|}\sum_{\substack{I^{'},I^{''}\in\mathcal{D},I^{'},I^{''}
\subset I\\l(I^{'})=2^{-m}l(I),l(I^{''})=2^{-n}l(I)}}
\langle f,h_{I^{'}}\rangle h_{I^{''}}\bigg|
\\ 
&=\bigg|\frac{1}{|I|}\sum_{\substack{I^{'}\subset I\\l(I^{'})=2^{-m}l(I)}}
\sum_{J\in \{ch(I^{'})\}}\alpha_{J}\int_{J}f(y)dyh_{I^{''}}\bigg|
\\ 
&\leq\sum_{\substack{I^{'}\subset I\\l(I^{'})=2^{-m}l(I)}}\frac{1}{|I|}
\bigg|\sum_{J\in \{ch(I^{'})\}}\frac{\alpha_{J}}{\|h_{I'}\|_{\infty}}
\int_{J}f(y)dy\bigg|
\end{align*}

Let $x_0\in I^{'}$ be a fixed point.  It is easy to see that
$\sum_{J\in \{ch(I^{'})\}}\alpha_{J}f(x_0)|J|=0$. Then,  the mean value 
theorem gives that
\begin{align*}
\biggl|\sum_{J\in \{ch(I^{'})\}}\frac{\alpha_{J}}{\|h_{I'}\|_{\infty}}
\int_{J}f(y)dy\biggr|
&=\biggl|\sum_{J\in \{ch(I^{'})\}}\frac{\alpha_{J}}{\|h_{I'}\|_{\infty}}
\int_{J}(f(y)-f(x_0))dy\biggr|
\\
&\le
\sum_{J\in \{ch(I^{'})\}}\sqrt d\, l(I')\|\nabla f\|_{\infty}|J|
\\
&\le 2^d \sqrt d\,2^{-m(d+1)}l(I)|I|.
\end{align*}
Consequently, this leads to 
\begin{align*}
\bigg|\frac{1}{|I|}\int_{I}a_{I}(x,y)f(y)dy\bigg|
&\leq\sum_{\substack{I^{'}\subset I\\l(I^{'})=2^{-m}l(I)}}\frac{1}{|I|}
2^d \sqrt d\,2^{-m(d+1)}l(I)|I|
&\leq 2^{-m}\sqrt{d}\,\|\nabla f\|_{\infty}l(I).
\end{align*}
Therefore, it holds that
\begin{align*}
{I_1}&\leq\bigg|\sum_{l(I)\leq 2^{-k_0}}\frac{1}{|I|}
\int_{I}a_{I}(x+t,y)f(y)dy\bigg|+\bigg|\sum_{l(I)\leq 2^{-k_0}}\frac{1}{|I|}
\int_{I}a_{I}(x,y)f(y)dy\bigg|
\\
&\leq 2^{-m+1}\sqrt{d}\,\|\nabla f\|_{\infty}\sum_{l(I)\leq 2^{-k_0}}l(I)
\\
&\lesssim\varepsilon.
\end{align*}
\noindent (2) Estimates for $I_2$.
Let $\widetilde{\mathcal{D}}$ consist of all the boundary points of the dyadic 
cubes $I\in\mathcal{D}$. 
Let $x\in\mathbb{R}^{d}\setminus\widetilde{\mathcal{D}}$. Then there exists 
$I_{k_0}\in\mathcal{D}_{-k_0-m}$ such that $x\in I_{k_0}$. 
If $I\in \cup_{k=-k_0+1}^{\infty}\mathcal D_k$ contains $x$, then it follows 
that 
$x\in I_{k_0}\subset I$ and $I$ is an $\ell$-th ancestor of $I_{k_0}$ for 
$\ell\ge m$. Hence $I_{k_0}$ is contained in one of $\mathit{ch}(I)$, which 
implies that $h_{I^{''}}(x+t)=h_{I^{''}}(x)$  for all 
 $I\in\bigcup_{k=-k_0+1}^{\infty}\mathcal{D}_{k}$. Thus, it follows that
\begin{align*}
{I_2}&=\lim_{t\rightarrow0}\Big|\sum_{l(I)> 2^{-k_0}}\frac{1}{|I|}\int_{I}\big(a_I(x+t,y)-a_I(x,y)\big)f(y)dy\Big|\\
&\leq\lim_{t\rightarrow0}\sum_{l(I)> 2^{-k_0}}\frac{1}{|I|}\Big|\sum_{\substack{I^{'},I^{''}\in\mathcal{D},I^{'},I^{''}\subset I\\l(I^{'})=2^{-m}l(I),l(I^{''})=2^{-n}l(I)}}\langle f,h_{I^{'}}\rangle \big(h_{I^{''}}(x+t)-h_{I^{''}}(x)\big)\Big|\\
&=\lim_{t\rightarrow0}\sum_{l(I)> 2^{-k_0}}\frac{1}{|I|}\bigg|\sum_{\substack{I^{'}\subset I\\l(I^{'})=2^{-m}l(I)}}\langle f,h_{I^{'}}\rangle\big(h_{I^{''}}(x+t)-h_{I^{''}}(x)\big)\bigg|\\
&=0.
\end{align*}
\noindent Therefore, $\mathbb Sf(x)$ is continuous almost everywhere.

\vspace{0.4cm}
\noindent \textbf{(ii)} 
Now, we consider the continuity of  $\pi_{b}^{\vec{\alpha}}(\vec{f})$. The proof of 
continuity for $T_{\epsilon}^{\vec{\alpha}}(\vec{f})$ follows similarly.
Let $\vec{\alpha}\in \{0,1\}^m\setminus\{(1,\cdots,1)\}$. Suppose that 
$f_j^{'}$ is bounded when $\alpha_j=0$ and $f_j$ is bounded when $\alpha_j=1$ 
in $\mathbb{R}$.
For $\varepsilon>0$, there exists $k_0>0$ such that 
$\sum_{k=k_0}^{\infty}\frac{1}{2^k}<\varepsilon$. Then, it holds that
\begin{align*}
\lim_{t\rightarrow0}|\pi_{b}^{\vec{\alpha}}(\vec{f})(x+t)
-\pi_{b}^{\vec{\alpha}}(\vec{f})(x)|\leq {II_1}+{II_2}
\end{align*}
where
\begin{align*}
{II_1}=\lim_{t\rightarrow0}\Big|&\sum_{l(I)\leq 2^{-k_0}}
\frac{\langle b,h_I\rangle}{|I|}\prod_{j=1}^{m}
\frac{\langle f_j, h_{I}^{1+\alpha_{j}}\rangle}{|I|} 
h_{I}^{1+\sigma(\vec{\alpha})}(x+t)
\\
&-\sum_{l(I)\leq 2^{-k_0}}\frac{\langle b,h_I\rangle}{|I|}
\prod_{j=1}^{m}\frac{\langle f_j, h_{I}^{1+\alpha_{j}}\rangle}{|I|} 
h_{I}^{1+\sigma(\vec{\alpha})}(x)\Big|
\end{align*}
and
\begin{align*}
{II_2}=\lim_{t\rightarrow0}\Big|&\sum_{l(I)>2^{-k_0}}
\frac{\langle b,h_I\rangle}{|I|}\prod_{j=1}^{m}
\frac{\langle f_j, h_{I}^{1+\alpha_{j}}\rangle}{|I|} 
h_{I}^{1+\sigma(\vec{\alpha})}(x+t)
\\
&-\sum_{l(I)> 2^{-k_0}}\frac{\langle b,h_I\rangle}{|I|}
\prod_{j=1}^{m}\frac{\langle f_j, h_{I}^{1+\alpha_{j}}\rangle}{|I|} 
h_{I}^{1+\sigma(\vec{\alpha})}(x)\Big|.
\end{align*}

Next, we will estimate ${II_1}$ and ${II_2}$, respectively.\par
\vspace{0.3cm}
\noindent (1) Estimates for ${II_1}$. 
For any $\alpha_j=0$, the mean value theorem yields that
\begin{align*}
\frac{|\langle f_j,h_I\rangle|}{|I|}
=\frac{|\int_{I_{+}}(f_j(x)-f_j(x_I))dx-\int_{I_{-}}(f_j(x)-f(x_i))dx|}{|I|}
\leq {|I|}\sup_{x\in\mathbb{R}} |f_j^{'}(x)|,
\end{align*}
where $x_I$ is the center of the interval $I$. 
By the definition of $\textit{BMO}^d$, we know that
$\frac{\langle b,h_I\rangle}{|I|}$ is bounded. The boundedness of 
$\frac{\langle f_j,h_I^{2}\rangle}{|I|}
=\frac{\langle f_j,\chi_{I}\rangle}{|I|}$ follows from the boundedness of 
$f_j$ in $\mathbb{R}$. These basic facts yield that
\begin{align*}
{II_1}
&\leq\Big|\sum_{l(I)\leq 2^{-k_0}}\frac{\langle b,h_I\rangle}{|I|}
\prod_{j=1}^{m}\frac{\langle f_j, h_{I}^{1+\alpha_{j}}\rangle}{|I|} 
h_{I}^{1+\sigma(\vec{\alpha})}(x+t)\Big|
\\
&\quad+\Big|\sum_{l(I)\leq 2^{-k_0}}\frac{\langle b,h_I\rangle}{|I|}
\prod_{j=1}^{m}\frac{\langle f_j, h_{I}^{1+\alpha_{j}}\rangle}{|I|} 
h_{I}^{1+\sigma(\vec{\alpha})}(x)\Big|
\\
&\leq2\sum_{l(I)\leq 2^{-k_0}}\Big|\frac{\langle b,h_I\rangle}{|I|}
\prod_{\alpha_j=0}\frac{\langle f_j, h_{I}\rangle}{|I|}
\prod_{\alpha_j=1}\frac{\langle f_j, \chi_{I}\rangle}{|I|}\Big|
\\
&\lesssim2\big\|b\big\|_{\textit{BMO}}
\sum_{k=k_0}^{\infty}\Bigl(\frac{1}{2^{k}}\Bigr)^{\sigma(\vec{\alpha})}
\\
&\lesssim\varepsilon
\end{align*}

\vspace{0.3cm}
\noindent (2) Estimates for ${II_2}$.
Let $\widetilde{\mathcal{D}}$ consist of all end-points of the dyadic intervals
 $I\in\mathcal{D}$.
Let $x\in\widetilde{\mathcal{D}}{ }^{c}$. Then there exists 
$I_{k_0}\in\mathcal{D}_{-k_0}$ such that $x\in I_{k_0}$. If 
$I\in \cup_{k=-k_0+1}^{\infty}\mathcal D_k$ contains $x$, then $I_{k_0}$ is 
contained in either $I_+$ or $I_-$, which implies
$h_{I}(x+t)=h_{I}(x)$ for $|t|<\dist(x,{I_{k_0}}^c)$. Therefore, for 
$I\in\bigcup_{k=-k_0+1}^{\infty}\mathcal{D}_{k}$ and $x\in I$ we get 
$h_{I}(x+t)=h_{I}(x)$. Then 
$h_{I}^{1+\sigma(\vec{\alpha})}(x+t)-h_{I}^{1+\sigma(\vec{\alpha})}(x)=0$. 
Consequently, it holds that
\begin{equation*}
{II_2}=\lim_{t\rightarrow0}\Big|\sum_{l(I)> 2^{-k_0}}
\frac{\langle b,h_I\rangle}{|I|}
\prod_{j=1}^{m}\frac{\langle f_j, h_{I}^{1+\alpha_{j}}\rangle}{|I|}
\bigg(h_{I}^{1+\sigma(\vec{\alpha})}(x+t)-h_{I}^{1+\sigma(\vec{\alpha})}(x)
\bigg)\Big| = 0. 
\end{equation*}
Finally, for $\vec{\alpha}\in \{0,1\}^m\setminus\{(1,\cdots,1)\}$, we have showed that
$\pi_{b}^{\vec{\alpha}}(\vec{f})(x)$ is continuous almost everywhere. 
When $\vec{\alpha}=(1,1,\cdots,1)$, let $b^{'}(x)$ be bounded, proceeding similar arguments as before, one may obtain that $\pi_{b}^{\vec{\alpha}}(\vec{f})(x)$ 
is almost everywhere continuous for all bounded 
$f_j$ in $\mathbb{R}$.
\qed
\par\vspace{0.5cm}
\section {Proofs of Theorems \ref{noncompaceness} and \ref{compactness-pi}}
 \label{non-c}

\begin{proof}[Proof of Theorem \ref{noncompaceness}]

Let $T$ be any of these dyadic operators and $\mathcal{K}=\{T(\vec{f})(x):\|f_j\|_{L^{p_j}}\leq1, j=1,\cdots,m\}$. According to the definition of compact operator, we need to show that $\mathcal{K}$ is precompact ($\overline{\mathcal{K}}$ is compact). It is obviously that
$T_{\{\epsilon_I=1\}}^{\{0\}}$ is the identity operator on $L^p(\mathbb R)$ by the reason that $\sum_{I\in\mathcal{D}}\langle f,h_I\rangle h_I=f(x)$ for 
$f\in L^p(\mathbb R)$ $(1<p<\infty)$. 
Moreover,
$T_{\{\epsilon_I=1\}}^{\{0\}}$ is not a compact operator since the unit ball of $L^p(\mathbb R)$ is not a compact set. 
Counter-examples will be given to illustrate that $\mathcal{K}$ doesn't satisfy 
the condition (c) for any Haar multipliers and dyadic shift, which implies the 
noncompactness of these dyadic operators. 
(i) By the Fr\'{e}chet-Kolmogorov-Riesz-Tsuji theorem, we need to show that
\begin{align*}
\mathcal{K}=\{T_{\epsilon}^{\vec{\alpha}}(\vec{f})(x): \|f_j\|_{p_j}\leq1\}
\end{align*}
at least does't meet one of the three conditions. 

We first observe the following:
For $I\in \mathcal D$, we define $\vec f_I=(f_{I,1},\dots, f_{I.m})$ by 
$f_{I,j}=|I|^{-1/p_j}h_I^{1+\alpha_j}$. Then we have
\begin{align}
&\|f_{I,j}\|_{p_j}=1,\label{eq:T-noncpt-1}
\\
&\langle f_{I,j}, h_I^{1+\alpha_j}\rangle=|I|^{1-1/p_j},\notag
\\
&\langle f_{I,j}, h_J \rangle=0,\ \text{ for }I\ne J\in\mathcal D, 
\alpha_j=0.\notag
\end{align}
Hence, noting $\alpha_j=0$ for at least one $1\le j\le m$, we get
\begin{equation*}
T_{\epsilon}^{\vec \alpha}(\vec f_I)
=\epsilon_I|I|^{-1/p}h_I^{\sigma(\vec\alpha)},
\end{equation*}
and so
\begin{equation}\label{eq:T-noncpt-2}
\|T_{\epsilon}^{\vec \alpha}(\vec f_I)\|_p=|\epsilon_I|.
\end{equation}
For $|I|<t$, we have $(I+t)\cap I=\emptyset$, and hence
\begin{equation}\label{eq:T-noncpt-3}
\|T_{\epsilon}^{\vec \alpha}(\vec f_I)(x+t)
-T_{\epsilon}^{\vec \alpha}(\vec f_I)(x)\|_p
=\|T_{\epsilon}^{\vec \alpha}(\vec f_I)\|_p=2|\epsilon_I|.
\end{equation}
Next, suppose that there exists $A>0$ such that 
$\#\{I\in\mathcal D: |\epsilon_I|\ge A\}=\infty$. We consider the following 
two cases: 
(1) $A_1:=\lim_{k\to\infty}\sup\limits_{I\in\mathcal D, 
I\subset [2^k,\infty)\cup (-\infty,-2^k)}|\epsilon_I|>0$, and 
(2) $\lim_{k\to\infty}\sup\limits_{I\in\mathcal D,
I\subset [2^k,\infty)\cup (-\infty,-2^k)}|\epsilon_I|=0$. \par
(1)  In this case, by \eqref{eq:T-noncpt-1} and \eqref{eq:T-noncpt-2} 
we see that 
\begin{equation*}
\limsup_{B\to\infty}\sup_{\|f_j\|_{p_j}\le1, 1\le j\le m}\biggl(
\int_{|x|\ge B}|T_{\epsilon}^{\vec \alpha}(\vec f_I)(x)|_pdx\biggr)^{1/p}
\ge A_1>0,
\end{equation*}
which shows the condition (b) does not hold.

(2) In this case, there exists $k_0\in\mathbb N$ such that  
$\#\{I\in \mathcal D, I\in [-2^{k_0},\,2^{k_0}]: |\epsilon_I|\ge A\}=\infty$, 
from which it follows that there exists $I_k\in\mathcal D$ such that 
$|\epsilon_{I_k}|\ge A, I_k\subset [-2^{k_0},\,2^{k_0}]$ 
and $\lim_{k\to\infty}|I_k|=0$. Hence, by \eqref{eq:T-noncpt-1} and 
\eqref{eq:T-noncpt-3}, it follows that
\begin{equation*}
\limsup_{t\to 0}\sup_{\|f_j\|_{p_j}\le1, 1\le j\le m}
\|T_{\epsilon}^{\vec \alpha}(\vec f)(x+t)
-T_{\epsilon}^{\vec \alpha}(\vec f)(x)\|_p\ge 2A>0.
\end{equation*}
This shows that condition (c) does not hold.

Hence, in any case, by Fr\'{e}chet-Kolmogorov-Riesz-Tsuji theorem, we know that
 $T_{\epsilon}^{\vec{\alpha}}(\vec{f})$ is not compact, under our assumption.

\vspace{0.3cm}
(ii) 
Suppose that $\mathbb S$ is a dyadic shift with  parameter $(m,n)$. 
Then, we can show that dyadic shift operator is not compact in the same way as
 in the case of $T_{\epsilon}^{\vec\alpha}$. We omit the proof of it. 
\end{proof}

\vspace{0.5cm} 
\begin{proof}[Proof of Theorem \ref{compactness-pi}]
 By the boundedness of $\pi_b^{\vec\alpha}$, it is trivial that 
$\pi_b^{\vec\alpha}$ satisfies the condition (a). Now we verify the condition 
(b) and the condition  (c) for its compactness.
We may assume $b\in C_c^\infty(\mathbb R)$ with $\supp b\subset (-1,1)$. For $k\ge1$,
The supports of $b$ and $h_I$ gives that
\begin{align*}
&\int_{|x|\ge2^k}|\pi_b^{\vec\alpha}(\vec f)(x)|^pdx
\\
&=\int_{|x|\ge2^k}\biggl|\sum_{I\in\mathcal D}\frac{\langle b, h_I\rangle}{|I|}
\prod_{j=1}^{m}\frac{\langle f_j, h_I^{1+\alpha_j}\rangle}{|I|}
h_I^{1+\sigma(\vec\alpha)}(x) \biggr|^p dx
\\
&=\int_{|x|\ge2^k}\biggl|\sum_{I=[0,2^\ell), [-2^\ell,0),\,\ell\ge k}
\frac{\langle b, h_I\rangle}{|I|} \prod_{j=1}^{m}\frac{\langle f_j, h_I^{1+\alpha_j}\rangle}{|I|}
h_I^{1+\sigma(\vec\alpha)}(x) \biggr|^p dx
\\
&\le \int_{|x|\ge2^k}\biggl(\sum_{I=[0,2^\ell), [-2^\ell,0),\,\ell\ge k}
\|b\|_{\infty} \frac{2}{|I|}
\prod_{j=1}^{m}\|f_j\|_{p_j}\frac{|I|^{1/p'_j}}{|I|}
\chi_{I}(x)\biggr)^pdx
\\
&\le C\Bigl(\|b\|_{\infty}\prod_{1}^m\|f_j\|_{p_j}\Bigr)^p
\int_{|x|\ge2^k}\biggl(\sum_{I=[0,2^\ell), [-2^\ell,0),\,\ell\ge k}
{|I|^{-1/p-1}}{\chi_{I}(x)}\biggr)^pdx
\\
&\le C\Bigl(\|b\|_{\infty}\prod_{1}^m\|f_j\|_{p_j}\Bigr)^p
\int_{|x|\ge2^k}\biggl(\sum_{\ell\ge k}
2^{-\ell (1/p+1)}{\chi_{[-2^\ell,2^\ell)}(x)}\biggr)^pdx
\\
&\le C\Bigl(\|b\|_{\infty}\prod_{1}^m\|f_j\|_{p_j}\Bigr)^p
\sum_{\ell=k}^{\infty}\int_{2^\ell}^{2^{\ell+1}}
\biggl(2^{-\ell (1/p+1)}\biggr)^p dx
\\
&= C\Bigl(\|b\|_{\infty}\prod_{1}^m\|f_j\|_{p_j}\Bigr)^p
\sum_{\ell=k}^{\infty}2^{-p\ell}
=C\Bigl(\|b\|_{\infty}\prod_{1}^m\|f_j\|_{p_j}\Bigr)^p2^{-pk}.
\end{align*}
Hence we have
\begin{equation*}
\lim_{A\to\infty}\int_{|x|\ge A}|\pi_b^{\vec\alpha}(\vec f)(x)|^p dx=0,
\end{equation*}
uniformly for $\vec f$ with $\|f_j\|_{p_j}\le 1$ $(1\le j\le m)$. Consequently, when $b\in \textit{CMO}$, $\pi_b^{\vec\alpha}$ satisfies the condition (b) for its compactness.

Let $1<p_1,\dots,p_m<\infty$ and $1/p=1/p_1+\dots+1/p_m$.
Now, we only need to consider dyadic intervals $I$
with $(-1,1)\cap I\ne\emptyset$ in the following summation. Therefore, it holds that
\begin{align*}
&\biggl(\int\biggl|\frac{\langle b,h_I \rangle}{|I|}
\prod_{j=1}^{m}\frac{\langle f_j, h_I^{1+\alpha_j}\rangle}{|I|}
(h_I^{1+\sigma(\vec\alpha)}(x+h)-h_I^{1+\sigma(\vec\alpha)}(x))\biggr|^pdx
\biggr)^{1/p}
\\
&\leq\biggl(\int\|b\|_{\infty}
\prod_{j=1}^{m}\frac{\|f_j\|_{p_j}|I|^{1/p'_j}}{|I|}
|(h_I^{1+\sigma(\vec\alpha)}(x+h)-h_I^{1+\sigma(\vec\alpha)}(x))|^pdx
\biggr)^{1/p}
\\
&\le C\|b\|_{\infty}\prod_{j=1}^{m}\|f_j\|_{p_j}
\frac{|h|^{1/p}}{|I|^{1/p}}.
\end{align*}
Thus we get
\begin{align*}
&\biggl(\int\Bigl|\sum_{|I|\ge1}\frac{\langle b,h_I \rangle}{|I|}
\prod_{j=1}^{m}\frac{\langle f_j, h_I^{1+\alpha_j}\rangle}{|I|}
(h_I^{1+\sigma(\vec\alpha)}(x+h)-h_I^{1+\sigma(\vec\alpha)}(x))\Bigr|^pdx
\biggr)^{1/p}
\\
&\le \sum_{|I|\ge1}\biggl(\int\Bigl|\frac{\langle b,h_I \rangle}{|I|}
\prod_{j=1}^{m}\frac{\langle f_j, h_I^{1+\alpha_j}\rangle}{|I|}
(h_I^{1+\sigma(\vec\alpha)}(x+h)-h_I^{1+\sigma(\vec\alpha)}(x))\Bigr|^pdx
\biggr)^{1/p}
\\
&\le
C\|b\|_{\infty}\prod_{j=1}^{m}\|f_j\|_{p_j}
\sum_{I=[0,2^\ell), [-2^\ell,0), \ell\in\mathbb N}
\frac{|h|^{1/p}}{|I|^{1/p}}
\le C\|b\|_{\infty}\prod_{j=1}^{m}\|f_j\|_{p_j}{|h|^{1/p}}.
\end{align*}

Next, for $|h|\le|I|$
, noting that $1/p=1/p_1+\dots+1/p_m$ and $\int h_I\,dx=0$, we have \begin{align*}
&\biggl(\int\Bigl|\frac{\langle b,h_I \rangle}{|I|}
\prod_{j=1}^{m}\frac{\langle f_j, h_I^{1+\alpha_j}\rangle}{|I|}
(h_I^{1+\sigma(\vec\alpha)}(x+h)-h_I^{1+\sigma(\vec\alpha)}(x))\Bigr|^pdx
\biggr)^{1/p}
\\
&=\biggl(\int\Bigl|\frac{\langle b-b(x_I),h_I \rangle}{|I|}
\prod_{j=1}^{m}\frac{\langle f_j, h_I^{1+\alpha_j}\rangle}{|I|}
(h_I^{1+\sigma(\vec\alpha)}(x+h)-h_I^{1+\sigma(\vec\alpha)}(x))\Bigr|^pdx
\biggr)^{1/p}
\\
&\le
\|b'\|_{\infty}|I| \prod_{j=1}^{m}\frac{|\langle f_j,h_I^{1+\alpha_j}\rangle|}{|I|}
\biggl(\int|h_I^{1+\sigma(\vec\alpha)}(x+h)-h_I^{1+\sigma(\vec\alpha)}(x)|^pdx
\biggr)^{1/p}
\\
&\le C\|b'\|_{\infty}
\prod_{j=1}^{m}\|f_j\chi_I\|_{p_j}{|h|^{1/p}}{|I|^{1-1/p}},
\end{align*}
where $x_I$ is the center of the dyadic interval $I$.

Similarly, in the case $|I|\le|h|$, it holds that$\bigl(\int|(h_I^{1+\sigma(\vec\alpha)}(x+h)
-h_I^{1+\sigma(\vec\alpha)}(x))|^pdx\bigr)^{1/p}\le C|I|^{1/p}$. Then, we may also obtain
\begin{align*}
&\biggl(\int\Bigl|\frac{\langle b,h_I \rangle}{|I|}
\prod_{j=1}^{m}\frac{\langle f_j, h_I^{1+\alpha_j}\rangle}{|I|}
(h_I^{1+\sigma(\vec\alpha)}(x+h)-h_I^{1+\sigma(\vec\alpha)}(x))\Bigr|^pdx
\biggr)^{1/p}
\\
&\le C\|b'\|_{\infty}
\prod_{j=1}^{m}\|f_j\chi_I\|_{p_j}{|I|},
\end{align*}
So, for any $0<a<1$, we have
\begin{align*}
&\biggl(\int\Bigl|\frac{\langle b,h_I \rangle}{|I|}
\prod_{j=1}^{m}\frac{\langle f_j, h_I^{1+\alpha_j}\rangle}{|I|}
(h_I^{1+\sigma(\vec\alpha)}(x+h)-h_I^{1+\sigma(\vec\alpha)}(x))\Bigr|^pdx
\biggr)^{1/p}
\\
&\le C\|b'\|_{\infty}
\prod_{j=1}^{m}\|f_j\chi_I\|_{p_j}{|h|^a}{|I|^{1-a}}.
\end{align*}
Thus, when $p>1$, for every $\ell\in\mathbb N$, we get 
\begin{align*}
&\sum_{|I|=2^{-\ell}}
\biggl(\int\Bigl|\frac{\langle b,h_I \rangle}{|I|}
\prod_{j=1}^{m}\frac{\langle f_j, h_I^{1+\alpha_j}\rangle}{|I|}
(h_I^{1+\sigma(\vec\alpha)}(x+h)-h_I^{1+\sigma(\vec\alpha)}(x))\Bigr|^pdx
\biggr)^{1/p}
\\
&\le
C\|b'\|_{\infty}{|h|^{1/p}}\sum_{|I|=2^\ell}
{|I|^{1-1/p}}\prod_{j=1}^{m}\|f_j\chi_I\|_{p_j}
\\
&\le
C\|b'\|_{\infty}{|h|^{1/p}}{2^{-(1-1/p)\ell}}
\prod_{j=1}^{m}\biggl(\sum_{|I|=2^{-\ell}}
\|f_j\chi_I\|_{p_j}^{p_j}\biggr)^{1/p_j}
\\
&\le C\|b'\|_{\infty}{|h|^{1/p}}{2^{-(1-1/p)\ell}}
\prod_{j=1}^{m}\|f_j\|_{p_j}.
\end{align*}
This leads to the following estimate:
\begin{align*}
&\biggl(\int\biggl|\sum_{|I|<1,\,I\cap(-1,1)\ne\emptyset}
\frac{\langle b,h_I \rangle}{|I|}
\prod_{j=1}^{m}\frac{\langle f_j, h_I^{1+\alpha_j}\rangle}{|I|}
(h_I^{1+\sigma(\vec\alpha)}(x+h)-h_I^{1+\sigma(\vec\alpha)}(x))
\biggr|^pdx\biggr)^{1/p}
\\
&\le
C\|b'\|_{\infty}\prod_{j=1}^{m}\|f_j\|_{p_j}{|h|^{1/p}}.
\end{align*}
When $p\le 1$ and $|h|<|I|$, it is easy to see that 
$|h|^{1/p}|I|^{1-1/p}<|h|^a|I|^{1-a}$ for some $0<a<1$.
Therefore, when $|h|<|I|<1$, for some $0<a<1$, we have
\begin{align*}
&\biggl(\int\Bigl|\frac{\langle b,h_I \rangle}{|I|}
\prod_{j=1}^{m}\frac{\langle f_j, h_I^{1+\alpha_j}\rangle}{|I|}
(h_I^{1+\sigma(\vec\alpha)}(x+h)-h_I^{1+\sigma(\vec\alpha)}(x))\Bigr|^pdx
\biggr)^{1/p}
\\
&\le C\|b'\|_{\infty}
\prod_{j=1}^{m}\|f_j\chi_I\|_{p_j}{|h|^a}{|I|^{1-a}}.
\end{align*}

Consequently, when $p\le 1$, by modifying a little bit, for some $0<a<1$, we get 
\begin{align*}
&\biggl(\int\biggl|\sum_{|I|<1,\,I\cap(-1,1)\ne\emptyset}
\frac{\langle b,h_I \rangle}{|I|}
\prod_{j=1}^{m}\frac{\langle f_j, h_I^{1+\alpha_j}\rangle}{|I|}
(h_I^{1+\sigma(\vec\alpha)}(x+h)-h_I^{1+\sigma(\vec\alpha)}(x))
\biggr|^pdx\biggr)^{1/p}
\\
&\le
C\|b'\|_{\infty}\prod_{j=1}^{m}\|f_j\|_{p_j}{|h|^{a}}.
\end{align*}
Thus, we obtain
\begin{equation*}
\lim_{h\to 0}\|{\pi}_b^{\vec\alpha}(x+h)-{\pi}_b^{\vec\alpha}(x)\|_p=0
\end{equation*}
uniformly for $\vec f$ with $\|f_j\|_{p_j}\le 1$ $(j=1,\dots,m)$. 
This shows that $\pi_b^{\vec{\alpha}}$ satisfies the condition (c).

Hence, by Fr\'{e}chet-Kolmogorov-Riesz-Tsuji theorem, it follows that
$\pi_b^{\vec{\alpha}}(\vec{f})$ is a compact operator.
\end{proof}
\noindent{\bf {Remark 3.1.}} 
The condition that $b\in \textit{CMO}$ is necessary by the reason that there exists 
$b\in L^\infty \subset \textit{BMO}$ such that $\pi_b^{\vec{\alpha}}$ is not 
a compact operator. To show this, we will construct an example.
 Let $k_0\in\mathbb N$ and $t\in[2^{-k_0},3\cdot2^{-k_0+1})$. Suppose that
\begin{equation*}
b=\sum_{k=1}^{\infty}(-1)^{k}\chi_{[1-2/2^{k}+1/2^{k+1},1-1/2^{k})}.
\end{equation*}
and
\begin{equation*}
f(x)=f_{k_0}(x)=
\begin{cases}
-2^{k_0} \ \ \ &\text{when  } x\in \frac{1}{2^{k_0}}\big([0,1)+2^{k_0}-2\big),
\\
0   \ \ \ &\text{otherwise  }.
\end{cases}
\end{equation*}
We assume that $0^0=0$ and $f_j=f^{1-\alpha_j}$. Then one can verify that
\begin{align*}
\limsup_{t\to0}\sup_{\|f_j\|_{L^{p_j}}\le 1}
\|\pi_b^{\vec{\alpha}}(\vec{f})(x+t)-\pi_b^{\vec{\alpha}}(\vec{f})(x)\|_{L^p}
&\gtrsim1.
\end{align*}
\par\smallskip
\section {Proof of Theorem  \ref{compaceness-commutators}}\label{TH5}
\vspace{0.5cm}
To begin with, we need to consider the strong type
 boundedness of these commutators. From \cite{IK2}, we know that the 
commutators in the $j$-th entry are bounded from $L^{p_1}\times L^{p_2}\times\cdots\times L^{p_m}
\rightarrow L^{p}$, if $b\in \mathit{BMO}$. Naturally, we ought to study the boundedness of iterated commutators and we obtain 
the following lemmas.\par
\begin{lemma}[Weighted strong bounds for $T_{\epsilon,\Pi\textbf{b}}^{\vec{\alpha}}$]\label{stong}
Let $\vec{p}=(p_1, \cdots, p_m)$ with $\frac{1}{p}=\frac{1}{p_1}+\cdots+\frac{1}{p_m}$ and $1<p_1, \cdots, p_m<\infty$. Let $\vec{\alpha}\in \{0,1\}^m\setminus\{1, \cdots, 1\}$ and 
$\epsilon=\{\epsilon_I\}_{I\in\mathcal{D}}$ be bounded. Suppose that
$\vec{b}=(b_1, \cdots, b_m)\in (\textit{BMO}^d)^m$, $\vec{\omega}\in A_{\vec{p}}^d$ and $\nu_{\vec{\omega}}=\prod_{j=1}^m\omega_j^{p/p_j}$. Then there exists a constant C such that
\begin{align}
\left\|T_{\epsilon,\Pi\textbf{b}}^{\vec{\alpha}}\right\|_{L^p(\nu_{\vec{\omega}})}\leq C\prod_{j=1}^m\left\|b_j\right\|_{\textit{BMO}_d}\prod_{i=1}^m\left\|f_i\right\|_{L^{p_i}(\omega_i)},
\end{align}
\end{lemma}

\begin{lemma}[Weighted end-point estimate for $T_{\epsilon,\Pi\textbf{b}}^{\vec{\alpha}}$]\label{end-point}
Let $\vec{\alpha}\in \{0,1\}^m\setminus\{1, \cdots, 1\}$ and $\epsilon=\{\epsilon_I\}_{I\in\mathcal{D}}$ be bounded. Suppose $\vec{b}=(b_1, \cdots, b_m)\in \textit{BMO}_d^m$, $\vec{\omega}\in A_{(1, \cdots, 1)}^d$ and $\nu_{\vec{\omega}}=\prod_{j=1}^m\omega_j^{p/p_j}$. Then there exists a constant C such that
\begin{align}
\nu_{\vec{\omega}}(x\in\mathbb{R}: T_{\epsilon,\Pi\textbf{b}}^{\vec{\alpha}}(\vec{f})(x)>t^m)\leq C\left(\prod_{j=1}^m\Phi\left(\frac{|f_j(x)|}{t}\right)\omega_j(x)dx\right)^{\frac{1}{m}},
\end{align}
where $\Phi(t)=t(1+\log^+t)$ and $\Phi^{(m)}=\overbrace{\Phi\circ\cdots\circ\Phi}^m$.
\end{lemma}
The ideas and main steps of proofs for Lemmas \ref{stong} and \ref{end-point} 
are almost the same as in \cite{PPTT} and \cite{X}.
Moreover, Lemma 3.1 of \cite{IK2} makes the proofs more easier. Here we 
omit the proofs. 
\par\smallskip 
Now we return to the proof of Theorem  \ref{compaceness-commutators}.
\textbf{(i)} First, we shall prove the compactness of commutator 
$[b,\,T_{\epsilon}^{\vec{\alpha}}]_i$. 
By its boundedness, verification of condition (a) is trivial and we will only 
prove that $[b,\,T_{\epsilon}^{\vec{\alpha}}]_i$ satisfies conditions (b) and 
(c) for its compactness. Firstly, we may assume $b\in C_c^\infty(\mathbb R)$ 
with $\supp b\subset (-1,1)$ and $f_j\in L^{p_j}(\mathbb R)$ $(1<p_j<\infty)$.
For $k\ge1$, by the supports of $b$ and $h_I$, it holds that
\begin{align*}
&\int_{|x|\ge2^k}|[b,\,T_{\epsilon}^{\vec\alpha}]_i(\vec f)(x)|^pdx
\\
&=\int_{|x|\ge2^k}\biggl|\sum_{I=[0,2^\ell), [-2^\ell,0),\,\ell\ge1}
\epsilon_I \frac{\langle bf_i, h_I^{1+\alpha_i}\rangle}{|I|}
\prod_{1\le j\le m,\,j\ne i}\frac{\langle f_j, h_I^{1+\alpha_j}\rangle}{|I|}
h_I^{\sigma(\vec\alpha)}(x)\biggr|^pdx
\\
&\le C\int_{|x|\ge2^k}\biggl(\sum_{I=[0,2^\ell), [-2^\ell,0),\,\ell\ge k}
|\epsilon_I| \|b\|_{\infty}
\frac{\|f_i\chi_{(-1,1)}\|_{p_i}}{|I|}
\prod_{1\le j\le m,\,j\ne i}\|f_j\|_{p_j}
\frac{|I|^{1/p'_j}}{|I|}
\chi_{I}(x)\biggr)^pdx
\\
&\le C\Bigl(\|\epsilon\|_{\infty}\|b\|_{\infty}\prod_{j=1}^m\|f_j\|_{p_j}\Bigr)^p
\int_{|x|\ge2^k}\biggl(\sum_{I=[0,2^\ell), [-2^\ell,0),\,\ell\ge k}
{|I|^{-1/p-1+1/p_i}}{\chi_{I}(x)}\biggr)^pdx
\\
&\le C\Bigl(\|\epsilon\|_{\infty}\|b\|_{\infty}\prod_{j=1}^m\|f_j\|_{p_j}\Bigr)^p
\int_{|x|\ge2^k}\biggl(\sum_{\ell\ge k}
2^{-\ell (1/p+1/p'_i)}{\chi_{[-2^\ell,2^\ell)}(x)}\biggr)^pdx
\\
&\le C\Bigl(\|\epsilon\|_{\infty}\|b\|_{\infty}\prod_{j=1}^m\|f_j\|_{p_j}\Bigr)^p
\sum_{\ell=k}^{\infty}\int_{2^\ell}^{2^{\ell+1}}
\biggl(2^{-\ell (1/p+1/p'_i)}\biggr)^pdx
\\
&= C\Bigl(\|\epsilon\|_{\infty}\|b\|_{\infty}\prod_{j=1}^m\|f_j\|_{p_j}\Bigr)^p
\sum_{\ell=k}^{\infty}2^{-\ell p/p'_i}
\\
&=C
\Bigl(\|\epsilon\|_{\infty}\|b\|_{\infty}\prod_{j=1}^m\|f_j\|_{p_j}\Bigr)^p
2^{-k p/p'_i}.
\end{align*}
Hence we have
\begin{equation*}
\lim_{A\to\infty}
\int_{|x|\ge A}|[b,\,T_{\epsilon}^{\vec\alpha}]_i(\vec f)(x)|^pdx=0,
\end{equation*}
uniformly for $\vec f$ with $\|f_j\|_{p_j}\le 1$ $(1\le j\le m)$. Consequently, when $b\in \textit{CMO}$, $[b,\,T_{\epsilon}^{\vec{\alpha}}]_i$ satisfies the condition (b) for any $1\le i\le m$ and $\vec{\alpha}\in \{0,1\}^m\setminus\{1,\cdots,1\}$.

Let $|h|<1$. We can rewrite $[b,\,T_{\epsilon}^{\vec{\alpha}}]_i(\vec f)(x+h)-[b,\,T_{\epsilon}^{\vec{\alpha}}]_i(\vec f)(x)$ in the following way.
\begin{align*}
&[b,\,T_{\epsilon}^{\vec{\alpha}}]_i(\vec f)(x+h)-[b,\,T_{\epsilon}^{\vec{\alpha}}]_i(\vec f)(x)
\\
&=b(x+h)T_{\epsilon}^{\vec\alpha}(\vec f)(x+h)
-T_{\epsilon}^{\vec\alpha}(f_1,\dots,b\,f_i,f_{i+1},\dots,f_m)(x+h)
\\
&\hspace{2.5cm}
 -b(x)T_{\epsilon}^{\vec\alpha}(\vec f)(x)
 +T_{\epsilon}^{\vec\alpha}(f_1,\dots,b\,f_i,f_{i+1},\dots,f_m)(x)
\\
&=(b(x+h)-b(x))T_{\epsilon}^{\vec{\alpha}}(\vec f)(x+h)
\\
&+\sum_{I\in\mathcal D}\epsilon_I(b(x)-b(x_I))
\prod_{j=1}^{m}\frac{\langle f_j, h_I^{1+\alpha_j}\rangle}{|I|}
h_I^{\sigma(\vec\alpha)}(x+h)
\\
&+\sum_{I\in\mathcal D}\epsilon_I\frac{\langle (b(x_I)-b)f_i, h_I^{1+\alpha_i}\rangle}{|I|}
\prod_{1\le j\le m, j\ne i}\frac{\langle f_j, h_I^{1+\alpha_j}\rangle}{|I|}
h_I^{\sigma(\vec\alpha)}(x+h)
\\
&-\sum_{I\in\mathcal D}\epsilon_I(b(x)-b(x_I))
\prod_{j=1}^{m}\frac{\langle f_j, h_I^{1+\alpha_j}\rangle}{|I|}
h_I^{\sigma(\vec\alpha)}(x)
\\
&+\sum_{I\in\mathcal D}\frac{\epsilon_I\langle (b(x_I)-b)f_i, h_I^{1+\alpha_i}\rangle}{|I|}
\prod_{1\le j\le m, j\ne i}\frac{\langle f_j, h_I^{1+\alpha_j}\rangle}{|I|}
h_I^{\sigma(\vec\alpha)}(x)
\\
&=: \mathbb I_1+\mathbb I_2+\mathbb I_3 -\mathbb I_4 -\mathbb I_5,
\end{align*}
where $x_I$ is the center of the dyadic interval $I$.

For $p$ with $1/p=1/p_1+\dots+1/p_m$, by the boundedness of
$T_{\epsilon}^{\vec{\alpha}}$, we obtain
\begin{equation*}\label{eq:T-cpt-1}
\|\mathbb I_1\|_p\le C\|b'\|_{\infty}|h|\prod_{j=1}^{m}\|f_j\|_{p_j}.
\end{equation*}
Now, we estimate $\|\mathbb I_2-\mathbb I_4\|_p$. Similar as in the proof of 
Theorem \ref{compactness-pi}, we get
\begin{align*}
&\biggl(\int\biggl|\sum_{|I|\ge1,\,I\cap(-1,1)\ne\emptyset}\epsilon_I(b(x)-b(x_I))
\prod_{j=1}^{m}\frac{\langle f_j, h_I^{1+\alpha_j}\rangle}{|I|}
(h_I^{\sigma(\vec\alpha)}(x+h)-h_I^{\sigma(\vec\alpha)}(x))\biggr|^pdx
\biggr)^{1/p}
\\
&\le \sum_{|I|\ge1,\,I\cap(-1,1)\ne\emptyset}\biggl(\int\biggl|\epsilon_I(b(x)-b(x_I))
\prod_{j=1}^{m}\frac{\langle f_j, h_I^{1+\alpha_j}\rangle}{|I|}
(h_I^{\sigma(\vec\alpha)}(x+h)-h_I^{\sigma(\vec\alpha)}(x))\biggr|^pdx
\biggr)^{1/p}
\\
&\le C \sum_{|I|\ge1,\,I\cap(-1,1)\ne\emptyset}\|b\|_{\infty}\|\{\epsilon_I\}\|_{\infty}\prod_{j=1}^{m}\|f_j\|_{p_j}
\frac{|h|^{1/p}}{|I|^{1/p}}
\le C\|b\|_{\infty}\|\{\epsilon_I\}\|_{\infty}\prod_{j=1}^{m}\|f_j\|_{p_j}{|h|^{1/p}}.
\end{align*}
If $|h|<|I|$, we take any $a<\min\{1,1/p\}$ and have $|h|^{1/p}|I|^{1-1/p}<|h|^{a}|I|^{1-a}$. Then we obtain
\begin{align*}
&\biggl(\int|\epsilon_I(b(x)-b(x_I))\prod_{j=1}^{m}\frac{\langle f_j, h_I^{1+\alpha_j}\rangle}{|I|}
(h_I^{\sigma(\vec\alpha)}(x+h)-h_I^{\sigma(\vec\alpha)}(x))|^pdx\biggr)^{1/p}
\\
&\le
\|b'\|_{\infty}\|\{\epsilon_I\}\|_{\infty}|I| \prod_{j=1}^{m}|\frac{\langle f_j,h_I^{1+\alpha_j}\rangle}{|I|} |
\biggl(\int|h_I^{\sigma(\vec\alpha)}(x+h)-h_I^{\sigma(\vec\alpha)}(x)|^pdx
\biggr)^{1/p}
\\
&\le C\|b'\|_{\infty}\|\{\epsilon_I\}\|_{\infty}
\prod_{j=1}^{m}\|f_j\chi_I\|_{p_j}{|h|^{1/p}}{|I|^{1-1/p}}
\\
&\le C\|b'\|_{\infty}\|\{\epsilon_I\}\|_{\infty}
\prod_{j=1}^{m}\|f_j\chi_I\|_{p_j}{|h|^{a}}{|I|^{1-a}}.
\end{align*}
When $|h|\ge|I|$, as in proof of Theorem \ref{compactness-pi}, we have
\begin{align*}
&\biggl(\int|\epsilon_I(b(x)-b(x_I))\prod_{j=1}^{m}\frac{\langle f_j, h_I^{1+\alpha_j}\rangle}{|I|}
(h_I^{\sigma(\vec\alpha)}(x+h)-h_I^{\sigma(\vec\alpha)}(x))|^pdx\biggr)^{1/p}
\\
&\le C\|b'\|_{\infty}\|\{\epsilon_I\}\|_{\infty}
\prod_{j=1}^{m}\|f_j\chi_I\|_{p_j}{|I|}
\le C\|b'\|_{\infty}\|\{\epsilon_I\}\|_{\infty}
\prod_{j=1}^{m}\|f_j\chi_I\|_{p_j}{|h|^{a}}{|I|^{1-a}}.
\end{align*}

Thus, for every $\ell\in\mathbb N$, it holds that
\begin{align*}
&
\biggl(\int\biggl|\sum_{|I|=2^{-\ell}}\epsilon_I
(b(x)-b(x_I))\prod_{j=1}^{m}\frac{\langle f_j, h_I^{1+\alpha_j}\rangle}{|I|}
(h_I^{\sigma(\vec\alpha)}(x+h)-h_I^{\sigma(\vec\alpha)}(x))\biggr|^pdx
\biggr)^{1/p}
\\
&\le
C\|b'\|_{\infty}\|\{\epsilon_I\}\|_{\infty}{|h|^{a}}\sum_{|I|=2^\ell}
{|I|^{1-a}}\prod_{j=1}^{m}\|f_j\chi_I\|_{p_j}
\\
&\le C\|b'\|_{\infty}\|\{\epsilon_I\}\|_{\infty}{|h|^{a}}{2^{-(1-a)\ell}}
\prod_{j=1}^{m}\|f_j\|_{p_j}.
\end{align*}
This leads to
\begin{align*}
&
\biggl(\int\biggl|\sum_{|I|<1,\,I\cap(-1,1)\ne\emptyset}\epsilon_I
(b(x)-b(x_I))\prod_{j=1}^{m}\frac{\langle f_j, h_I^{1+\alpha_j}\rangle}{|I|}
(h_I^{\sigma(\vec\alpha)}(x+h)-h_I^{\sigma(\vec\alpha)}(x))\biggr|^pdx
\biggr)^{1/p}
\\
&\le
C\|b'\|_{\infty}\|\{\epsilon_I\}\|_{\infty}\prod_{j=1}^{m}\|f_j\|_{p_j}{|h|^{a}}.
\end{align*}
Then we have $\|\mathbb I_2-\mathbb I_4\|_p\le C\|b'\|_{\infty}\|\{\epsilon_I\}\|_{\infty}\prod_{j=1}^{m}\|f_j\|_{p_j}{|h|^{a}}$. The estimate of $\|\mathbb I_3-\mathbb I_5\|_p$ is similar and we omit the details.\par
Therefore, we have shown that
\begin{equation*}
\lim_{h\to 0}\|[b,\,T_{\epsilon}^{\vec{\alpha}}]_i(\vec f)(x+h)-[b,\,T_{\epsilon}^{\vec{\alpha}}]_i(\vec f)(x)\|_p=0
\end{equation*}
uniformly for $\vec f$ with $\|f_j\|_{p_j}\le 1$ $(j=1,\dots,m)$. \qed
\par\smallskip
\textbf{(ii)} Proof of compactness for iterated commutators. 
We will need the following lemma.
\begin{lemma}\label{lemma1}
Let T be a multilinear operator and $b\in \textit{CMO}$. Suppose that T is a compact operator from $L^{p_1}(\mathbb{R})\times\cdots\times L^{p_m}(\mathbb{R})$ to $L^p(\mathbb{R})$. Then for any $i$ with $1\le i\le m$, $[b,T]_i$ is a compact operator.
\end{lemma} 
\begin{proof}
To illustrate $[b,\,T]_i$ is a compact operator, we only need to verify 
conditions (a), (b) and (c). 
Let $\mathcal K=\{[b,\,T]_i(\vec{f}): \|f\|_{p_j}\le 1, j=1,\cdots,m\}$. 
We can deduce that $T$ is bounded operator because $T$ is a compact operator.
Therefore, we have
\begin{align*}
\|[b,\,T]_i(\vec{f})\|_p&\le \|b\|_{\infty}\|T(\vec{f})\|_p+\|T(f_1,\cdots,bf_i,\cdots,f_m)\|_p
\\
&\le \Big(\|b\|_{\infty}\prod_{j=1}^m\|f_j\|_{p_j}+\prod_{j\ne i}\|f_j\|_{p_j}\|bf_i\|_{p_i}\Big)
\\
&\le 2\|b\|_{\infty}\prod_{j=1}^m\|f_j\|_{p_j}\le 2\|b\|_{\infty},
\end{align*}
which implies that the condition (a) holds. To check the 
conditions (b) and (c). 
By the boundedness of 
$[b,\,T]_i(\vec{f})$, we may assume $b\in C_c^\infty(\mathbb R)$. Due to the compactness of $T$, 
by the Fr\'{e}chet-Kolmogorov-Riesz-Tsuji theorem, we have 
\begin{equation}\label{equation1}
\lim_{A\to\infty}\sup_{\|f_j\|_{p_j}\le 1}\int_{|x|\ge A}|T(\vec f)(x)|^p dx=0,
\end{equation}
\begin{equation}\label{equation2}
\lim_{h\to 0}\sup_{\|f_j\|_{p_j}\le 1}\|T(\vec f)(x+h)-T(\vec f)(x)\|_p=0.
\end{equation}
We assume that $M_g^i(\vec f)=(f_1,\cdots,gf_i,\cdots,f_m)$. 
From (\ref{equation1}), we see that
\begin{align*}
&\lim_{A\to\infty}
\sup_{\|f_j\|_{p_j}\le 1}\int_{|x|\ge A}|[b,\,T]_i(\vec f)(x)|^p dx
\\
&\lesssim\|b\|_{\infty}^p\lim_{A\to\infty}
\sup_{\|f_j\|_{p_j}\le 1}\Big(\int_{|x|\ge A}|T(\vec f)(x)|^p dx
+\int_{|x|\ge A}|T(M_{\frac{b}{\|b\|_{\infty}}}^i(\vec f))(x)|^p dx\Big)=0.
\end{align*}
Obviously, it follows that
\begin{align*}
&\lim_{h\to 0}
\sup_{\|f_j\|_{p_j}\le 1}\|[b,\,T]_i(\vec f)(x+h)-[b,\,T]_i(\vec f)(x)\|_p
\\
&\le \lim_{h\to 0
}\sup_{\|f_j\|_{p_j}\le 1}\Big(\|b(x+h)T(\vec f)(x+h)-b(x)T(\vec f)(x+h)\|_p
\\
&+\|b(x)T(\vec f)(x+h)-b(x)T(\vec f)(x)\|_p
+\|T(M_b^i(\vec f))(x+h)-T(M_b^i(\vec f))(x)\|_p\Big).
\end{align*}
By (\ref{equation2}) and the boundedness of $T$, we deduce that
\begin{align*}
&\lim_{h\to 0}\sup_{\|f_j\|_{p_j}\le 1}\|[b,\,T]_i(\vec f)(x+h)-[b,\,T]_i(\vec f)(x)\|_p
\\
&\le \|b'\|_{\infty}\lim_{h\to 0}
\sup_{\|f_j\|_{p_j}\le 1}|h|\|T(\vec f)(x+h)\|_p
+\|b\|_{\infty}\lim_{h\to 0}
\sup_{\|f_j\|_{p_j}\le 1}\|T(\vec f)(x+h)-T(\vec f)(x)\|_p
\\
&+\|b\|_{\infty}\lim_{h\to 0}\sup_{\|f_j\|_{p_j}\le 1}
\|T(M_{\frac{b}{\|b\|_{\infty}}}^i(\vec f))(x+h)
-T(M_{\frac{b}{\|b\|_{\infty}}}^i(\vec f))(x)\|_p=0.
\end{align*}
Therefore, $[b,\,T]_i$ is a compact operator for any $i$ with 
$1\le i\le m$.
\end{proof}
Now, by Lemma \ref{lemma1} and the compactness of commutator 
$[b,\,T_{\epsilon}^{\vec \alpha}]_i$, we can deduce that iterated commutators 
$T_{\epsilon,\Pi\textbf{b}}^{\vec{\alpha}}$ is a compact operator for all 
$\vec\alpha\in\{0,1\}^m\setminus\{1,\cdots,1\}$.
\\

\textbf{(iii)} Proof of the compactness of $[b,\mathbb S]$.
We need the following lemma for $[b,\mathbb S]$.
\begin{lemma}[]\label{boundedness-[b,S]}
Let $b\in\mathrm{BMO}(\mathbb R^d)$ and $1<p<\infty$. Then, 
for any $1<p<\infty$ there exists $C>0$ such that
\begin{equation*}
\|[b,\mathbb S](f)\|_p\le C\|b\|_{\mathrm{BMO}}\|f\|_p.
\end{equation*}
\end{lemma}
\begin{proof}
Let $x\in\mathbb R$ and $I_0\in\mathcal D$ contain $x$. 
Let $1<s<\infty$ and $I_0^{n}$ be the $n$-th ancestor of $I_0$. 
Then, it holds that
\begin{align*}
[b,\mathbb S](f)(x)&=(b(x)-b_{I_0})\mathbb Sf(x)-\mathbb S((b-b_{I_0})f\chi_{I_0^n})(x)-\mathbb S((b-b_{I_0})f\chi_{(I_0^n)^c})(x).
\end{align*}
It is easy to show that $$\frac{1}{|I_0|}\int_{I_0}|(b(x)-b_{I_0})\mathbb Sf(z)|dz\le C\|b\|_{\mathrm{BMO}}M_{s}^d(\mathbb S(f))(x)$$ and $$\frac{1}{|I_0|}\int_{I_0}|\mathbb S((b-b_{I_0})f\chi_{I_0^n})(z)|dz\le C\|b\|_{\mathrm{BMO}}M_{s}^d(f)(x).$$ As for $\mathbb S((b-b_{I_0})f\chi_{(I_0^n)^c})(x)$, take any $y\in I_0$ and denote $x_{I_{0}}$ to be the center of $I_0$. 
To estimate $\mathbb S((b-b_{I_0})f\chi_{(I_0^n)^c})(x)$ on $I_0$, we only need to treat cubes $I'$ and $I''$ with 
$I''\cap I_0\ne\emptyset$ and $I'\cap(I_0^n)^c\ne\emptyset$. 
Since $I''\subset I_0$ it follows that $I_0^n\supset (I'')^n=I$, which contradicts 
with $I'\cap(I_0^n)^c\ne\emptyset$. So we have $I_0\varsubsetneq  I''$. 
Since $h_{I''}$ is a constant on each child of $I''$, it is a constant on $I_0$. 
Hence we get
\begin{align*}
\mathbb S((b-b_{I_0})f\chi_{(I_0^n)^c})(y)&=
\sum_{I\in \mathcal{D}}\sum_{\substack{I^{'},I^{''}\in
\mathcal{D},I^{'},I^{''}\subset I\\l(I^{'})=2^{-m}l(I),l(I^{''})=2^{-n}l(I)}}
\frac{1}{|I|}
\langle (b-b_{I_0})\chi_{(I_0^{n})^c} f,h_{I^{'}}\rangle h_{I^{''}}(y)
\\
&=\mathbb S((b-b_{I_0})f\chi_{(I_0^n)^c}(x_{I_0}).
\end{align*}
Thus we have
\begin{equation}\label{MDS}
M^{\#}([b,\mathbb S](f))(x)
\le C\|b\|_{\mathrm{BMO}}
\bigl(M_{s}^d(\mathbb S(f))(x)+M_{s}^d(f)(x)\bigr).
\end{equation}
Let $b\in L^\infty(\mathbb R^d)$ and $f\in L^p(\mathbb R^d)$. 
Then, since $[b,\mathbb S](f)(x)=b(x)\mathbb S(f)(x)-\mathbb S(bf)(x)$, 
it follows that $[b,\mathbb S](f)\in L^p(\mathbb R^d)$. By inequality (\ref{MDS}), we see that 
\begin{equation*}
\|[b,\mathbb S](f)\|_p\le C\|b\|_{\mathrm{BMO}}\|f\|_p.
\end{equation*}
Now, letting 
\begin{equation*}
b_j(x)=\begin{cases}
j, &\text{if }b(x)> j,\\
b(x), &\text{if }|b(x)|\le j,\\
-j, &\text{if }b(x)< -j, 
\end{cases}
\end{equation*}
and taking a subsequence (if necessary), we can deduce that  (Cf. \cite{LOPTT}):
\begin{equation*}
\|[b,\mathbb S](f))\|_p\le C\|b\|_{\mathrm{BMO}}\|f\|_p.
\end{equation*}
\end{proof}
Now we turn to the proof of {(iii)}. 
We may assume $b\in C_c^\infty(\mathbb R^d)$ with $\supp b\subset (-1,1)^d$. 
For the sake of simplicity, we only consider the integration on 
$E_k:=[0,\infty)^d\setminus [0,2^k)^d$. Notice that
\begin{align*}
&\int_{E_k}|[b,\mathbb S](f)(x)|^pdx
\\
&=\int_{E_k}\biggl|\sum_{I=[0,2^\ell)^d, \,\ell\ge k}
\sum_{\substack{I^{'},I^{''}\in
\mathcal{D},I^{'},I^{''}\subset I\\l(I^{'})=2^{-m}l(I),l(I^{''})=2^{-n}l(I)}}
\frac{1}{|I|}
\langle bf,h_{I^{'}}\rangle  h_{I^{''}}(x)\biggr|^pdx
\\
&\le C
\int_{E_k}\biggl(\sum_{I=[0,2^\ell)^d, \,\ell\ge k}
\sum_{\substack{I^{'},I^{''}\in
\mathcal{D},I^{'},I^{''}\subset I\\l(I^{'})=2^{-m}l(I),l(I^{''})=2^{-n}l(I)}}
\frac{\|b\|_\infty\|f\chi_{[0,1)^d}\|_1\|h_{I'}\|_\infty\|h_{I''}\|_\infty}{|I|}
\chi_{I^{''}}(x)\biggr)^pdx
\\
&\le C\|b\|_\infty^p\|f\|_p^p
\int_{E_k}\biggl(\sum_{I=[0,2^\ell)^d, \,\ell\ge k}
|I|^{-1}\chi_{I}(x)\biggr)^pdx
\le C\|b\|_\infty^p\|f\|_p^p\times 2^{-(p-1)dk},
\end{align*}
which yields that
\begin{equation*}
\lim_{A\to\infty}
\int_{|x|\ge A}|[b,\mathbb S](f)(x)|^pdx=0,
\end{equation*}
uniformly for $ f$ with $\|f\|_{p}\le 1$. This shows that condition (b) holds. 
\smallskip

Now, we are ready to check condition (c). We rewrite $[b,\mathbb S](f)(x+h)-[b,\mathbb S](f)(x)$ 
in the following way
\begin{align*}
&[b,\mathbb S](f)(x+h)-[b,\mathbb S](f)(x)
\\
&= (b(x+h)-b(x))\mathbb S(f)(x+h)
\\
&\hspace{0.5cm}
+\sum_{I\in \mathcal{D}}\sum_{\substack{I^{'},I^{''}\in
\mathcal{D},I^{'},I^{''}\subset I\\l(I^{'})=2^{-m}l(I),l(I^{''})=2^{-n}l(I)}}
\frac{1}{|I|}
(b(x)-b(x_{I'}))\langle f,h_{I^{'}}\rangle (h_{I^{''}}(x+h)-h_{I^{''}}(x))
\\ 
&\hspace{0.51cm}-\sum_{I\in \mathcal{D}}\sum_{\substack{I^{'},I^{''}\in
\mathcal{D},I^{'},I^{''}\subset I\\l(I^{'})=2^{-m}l(I),l(I^{''})=2^{-n}l(I)}}
\frac{1}{|I|}
\langle (b-b(x_{I'}))f,h_{I^{'}}\rangle\bigl(h_{I^{''}}(x+h)-h_{I^{''}}(x)\bigr)
\\ 
&=:II_1+II_2+II_3.
\end{align*}
The $L^p$ boundedness of $\mathbb S$ yields that 
\begin{equation}\label{eq:SS-cpt-1}
\|II_1\|_p\le C\|\nabla b\|_{\infty}\|f\|_{p}|h|.
\end{equation}
For $I,I^{'},I^{''}\in \mathcal{D}$ satisfying 
$|I|\ge1,I\cap(-1,1)^d\ne\emptyset$ and $I',I''\subset I$, 
$l(I^{'})=2^{-m}l(I),l(I^{''})=2^{-n}l(I)$, we have
\begin{align*}
&\biggl(\int\biggl|
\frac{1}{|I|}(b(x)-b(x_{I'}))\langle f,h_{I^{'}}\rangle 
(h_{I^{''}}(x+h)-h_{I^{''}}(x))\biggr|^pdx\biggr)^{1/p}
\\
&\le C\|b\|_{\infty} |I|^{-1}\|f\chi_{I'}\|_{p}{|I'|^{1/p'}}\|h_{I'}\|_{\infty}
\|h_{I''}\|_{\infty}
\bigl({|h|}{|I|^{(d-1)/d}}\bigr)^{1/p}
\\
&\le C\|b\|_{\infty}\|f\|_{p}{|h|^{1/p}}{|I|^{-1/(dp)}}.
\end{align*}
Thus, noting that $\supp b\subset (-1,1)^d$, we get 
\begin{align*}
&\biggl(\int\biggl|\sum_{\substack{I\in \mathcal{D},|I|\ge1,\\
I\cap(-1,1)^d\ne\emptyset}}
\sum_{\substack{I^{'},I^{''}\in
\mathcal{D},I^{'},I^{''}\subset I\\l(I^{'})=2^{-m}l(I),l(I^{''})=2^{-n}l(I)}}\frac{1}{|I|}(b(x)-b(x_{I'}))\langle f,h_{I^{'}}\rangle \\&\quad\times
(h_{I^{''}}(x+h)-h_{I^{''}}(x))\biggr|^pdx\biggr)^{1/p}
\\
&\le 
C\|b\|_{\infty}\|f\|_{p}|h|^{1/p}\cdot
2^{(m+n)d}\sum_{I=[0,2^\ell)^d, \ell\in\mathbb N}|I|^{-1/(dp)}
\le C\|b\|_{\infty}\|f\|_{p}{|h|^{1/p}}.
\end{align*}
\par\smallskip
Next, we treat the case $|h|^d<|I|<1$. 
For every $\ell\in\mathbb N$, suing the support of 
$h_I^{\sigma(\vec\alpha)}(x+h)-h_I^{\sigma(\vec\alpha)}(x)$ and by H\"older's inequality, 
we get
\begin{align*}
&\biggl(\int\biggl|\sum_{\substack{I\in \mathcal{D},|h|^d<|I|=2^{-d\ell},\\
I\cap(-1,1)^d\ne\emptyset}}
\sum_{\substack{I^{'},I^{''}\in
\mathcal{D},I^{'},I^{''}\subset I\\l(I^{'})=2^{-m}l(I),l(I^{''})=2^{-n}l(I)}}
\frac{1}{|I|}(b(x)-b(x_{I'}))\\&\quad\times \langle f,h_{I^{'}}\rangle 
(h_{I^{''}}(x+h)-h_{I^{''}}(x))\biggr|^pdx\biggr)^{1/p}
\\
&\le C{\|\nabla b\|_{\infty}}
\biggl(\int\biggl|\sum_{\substack{I\in \mathcal{D},|h|^d<|I|=2^{-d\ell},\\
I\cap(-1,1)^d\ne\emptyset}}
\sum_{\substack{I^{''}\in
\mathcal{D},I^{''}\subset I\\l(I^{''})=2^{-n}l(I)}}
{|I|^{\frac1d-\frac1p}}\|f\chi_{I}\|_{p}
\frac{|h_{I^{''}}(x+h)-h_{I^{''}}(x)|}{\|h_{I^{''}}\|_{\infty}}\biggr|^pdx
\biggr)^{1/p}
\\
&\le C{\|\nabla b\|_{\infty}}
\biggl(\sum_{\substack{I\in \mathcal{D},|h|^d<|I|=2^{-d\ell},\\
I\cap(-1,1)^d\ne\emptyset}}
{|I|^{\frac{p}d-1}}\|f\chi_{I}\|_{p}^p|h||I|^{\frac{d-1}{d}}
\biggr)^{1/p}
\\ 
&\le 
 C2^{-\ell(1-1/p)}\|\nabla b\|_{\infty}\|f\|_{p}{|h|^{1/p}}.
\end{align*}
Thirdly, noting that $|I|^{1/d}<|h|$ implies $|h|/|I|^{1/d}>1$, we get

\begin{align*}
&\biggl(\int\biggl|\sum_{\substack{I\in \mathcal{D},|I|=2^{-d\ell}\le|h|^d<1,\\
I\cap(-1,1)^d\ne\emptyset}}
\sum_{\substack{I^{'},I^{''}\in
\mathcal{D},I^{'},I^{''}\subset I\\l(I^{'})=2^{-m}l(I),l(I^{''})=2^{-n}l(I)}}
\frac{1}{|I|}(b(x)-b(x_{I'}))\\&\quad \times\langle f,h_{I^{'}}\rangle 
(h_{I^{''}}(x+h)-h_{I^{''}}(x))\biggr|^pdx\biggr)^{1/p}
\\ 
&\le C{\|\nabla b\|_{\infty}}
\biggl(\sum_{\substack{I\in \mathcal{D},|I|=2^{-d\ell}\le|h|^d<1,\\
I\cap(-1,1)^d\ne\emptyset}}
|I|^{\frac{p}{d}}\frac{|h|}{|I|^{\frac1 d}}\|f\chi_{I}\|_{p}^p
\biggr)^{1/p}
\\ 
&\le C2^{-\ell(1-1/p)}{\|\nabla b\|_{\infty}}|h|^{1/p}
\biggl(\sum_{\substack{I\in \mathcal{D},|I|=2^{d\ell}\le|h|^d<1,\\
I\cap(-1,1)^d\ne\emptyset}}
\|f\chi_{I}\|_{p}^p
\biggr)^{1/p}
\\ 
&\le 
 C2^{-\ell(1-1/p)}\|\nabla b\|_{\infty}\|f\|_{p}{|h|^{1/p}}.
\end{align*}
Hence
\begin{align*}
&\biggl(\int\biggl|\sum_{\substack{I\in \mathcal{D},|I|<1,\\
I\cap(-1,1)^d\ne\emptyset}}
\sum_{\substack{I^{'},I^{''}\in
\mathcal{D},I^{'},I^{''}\subset I\\l(I^{'})=2^{-m}l(I),l(I^{''})=2^{-n}l(I)}}
\frac{1}{|I|}(b(x)-b(x_{I'}))\\&\quad \times \langle f,h_{I^{'}}\rangle 
(h_{I^{''}}(x+h)-h_{I^{''}}(x))\biggr|^pdx\biggr)^{1/p}
\\
&\le \sum_{\ell=1}^{\infty}C
2^{-\ell(1-1/p)}\|\nabla b\|_{\infty}\|f\|_{p}{|h|^{1/p}}
\\
&\le 
 C\|\nabla b\|_{\infty}\|f\|_{p}|h|^{1/p}.
\end{align*}
Thus, we obtain that $\|II_2\|_{p}\le C\|\nabla b\|_{\infty}\|f\|_{p}|h|^{1/p}$.
Similarly we have the same estimate for $\|II_3\|_{p}$. 


\par\smallskip 

Therefore, we have shown that $[b,\mathbb S]$ satisfies conditions (a)-(c) and $[b,\mathbb S]$ is a compact operator.
\qed


\end{document}